\documentclass[11pt]{amsart}

\usepackage[usenames,dvipsnames,svgnames,table]{xcolor} 
\usepackage{amssymb,amsfonts,amsrefs}
\usepackage{amsmath,amsthm}

\usepackage{pgf, tikz} 
\usepackage[active]{srcltx}		
\usepackage{pdfsync}					
\usepackage{enumitem}
\usepackage[colorinlistoftodos]{todonotes}
\usepackage{dsfont}
\newcommand{\dd}{\,{\rm d}}

\newcommand{\e}{\epsilon}

\newcommand{\s}[1]{{\mathcal #1}}

\newcommand{\bb}[1]{{\mathbb #1}}

\newtheorem{theorem}{Theorem}[section]
\newtheorem{proposition}[theorem]{Proposition}
\newtheorem{lemma}[theorem]{Lemma}
\newtheorem{corollary}[theorem]{Corollary}

\theoremstyle{definition}

\theoremstyle{remark}
\newtheorem{remark}[theorem]{Remark}
\numberwithin{equation}{section}

\begin{document}

\title[Variational mean field games for market competition]
{Variational mean field games for market competition}

\author{P. Jameson Graber}
\thanks{National Science Foundation under NSF Grant DMS-1612880.}
\address{P. J. Graber: Baylor University, Department of Mathematics;
One Bear Place \#97328;
Waco, TX 76798-7328 \\
Tel.: +1-254-710- \\
Fax: +1-254-710-3569 
}
\email{Jameson\_Graber@baylor.edu}

\author{Charafeddine Mouzouni}
\thanks{French National Research Agency under ANR-10-LABX-0070, ANR-11-IDEX-0007 and ANR-16-CE40-0015-01-MFG.}
\address{C. Mouzouni: Univ Lyon, \'Ecole centrale de Lyon, CNRS UMR 5208, Institut Camille Jordan, 36 avenue Guy de Collonge, F-69134 Ecully Cedex, France.}
\email{mouzouni{@}math.univ-lyon1.fr}

\dedicatory{Version: \today}

\begin{abstract}
In this paper, we explore Bertrand and Cournot Mean Field Games models for market competition with reflection boundary conditions. We prove existence, uniqueness and regularity of solutions to the system of equations, and show that this system can be written as an optimality condition of a convex minimization problem. 
We also provide a short proof of uniqueness to the system addressed in [Graber, P. and Bensoussan, A., \textit{Existence and uniqueness of solutions for Bertrand and Cournot mean field games,} Applied Mathematics \& Optimization (2016)], where uniqueness was only proved for small parameters $\e$.
Finally, we prove existence and uniqueness of a weak solutions to the corresponding first order system at the deterministic limit.
\end{abstract}


\maketitle

\section{Introduction} \label{sec:intro}
Our purpose is to study the following coupled system of partial differential equations:
\begin{equation}
\label{main system}
\left\{
\begin{array}{rcc}
(i)& u_t + \frac{\sigma^2}{2} u_{xx} - ru + G(u_x,m)^2 = 0,\quad & 0 < t < T, \ 0 < x < L\\
(ii)&m_t -  \frac{\sigma^2}{2} m_{xx} - \left\{G(u_x,m)m\right\}_x = 0, \quad & 0 < t < T, \ 0 < x < L\\
(iii) & m(0,x) = m_0(x), u(T,x) = u_T(x), \quad & 0 \leq x \leq L\\
(iv) & u_x(t,0) = u_x(t,L) = 0,\quad & 0 \leq t \leq T \\ 
(v) & \frac{\sigma^2}{2} m_x(t,x) + G(u_x,m)m(t,x) = 0,\quad & 0 \leq t \leq T, x \in \{0,L\}
\end{array}\right.
\end{equation}
where 
$
G(u_x,m) := \frac{1}{2}\left( b+c\int_{0}^{L} u_x(t,y)m(t,y)\dd y- u_{x}\right),
$
$\sigma,b,c,T,L$ are given positive constants, and $m_0(x),u_T(x)$ are known functions.

System \eqref{main system} is in the family of models introduced by Gu\'eant, Lasry, and Lions \cite{gueant2011mean} as well as by Chan and Sircar in \cite{chan2015bertrand,chan2015fracking} to describe a \emph{mean field game} in which producers compete to sell an exhaustible resource such as oil. 
The basic notion of mean field games (MFG) was introduced by Lasry and Lions \cite{lasry06,lasry06a,lasry07} and Caines, Huang, and Malham\'e \cite{huang2006large}.
Here we view the producers as a continuum of rational agents whose is given by the function $m(t,x)$ governed by a Fokker-Planck equation.
Each of them must solve an optimal control problem in order to optimize profit, which corresponds to the Hamilton-Jacobi-Bellman equation \eqref{main system}(i).
A solution to the coupled system therefore corresponds (formally) to a Nash equilibrium among infinitely many competitors in the market.

The analysis of this type of PDE system was already addressed in \cite{graber2016existence} with Dirichlet boundary conditions at $x=0$.
It is  a framework where producers have limited stock, and they leave the market as soon as their stock is exhausted. 
In particular, the density of players is a non-increasing function \cite{graber2016existence}.
By contrast, in studying system \eqref{main system} we explore a new boundary condition.
In terms of the model, we assume that players never leave the game so that the number of producers in the market remains constant. In this particular case, the density of players is a probability density for all the times, which considerably simplifies the analysis of the system of equations.
Further details on the interpretation of the problem will be given below in Section \ref{sec:specification}.

Applications of mean field games to economics have attracted much recent interest; see \cite{achdou2014partial,burger2014partial,gomes2014socio} for surveys of the topic.
Nevertheless, most results from the PDE literature on mean field games are not sufficient to establish well-posedness for models of market behavior such as \eqref{main system}.
In particular, many authors have studied existence and uniqueness of solutions to systems of the type
\begin{equation} \label{eq:mfg classical}
\begin{array}{c}
u_t + \frac{1}{2}\sigma^2 u_{xx} - ru + H(t,x,u_x) = V[m],\\
m_t - \frac{1}{2}\sigma^2 m_{xx} - \left(G(t,x,u_x)m\right)_x = 0.
\end{array}
\end{equation}
See, for example,  \cite{porretta2015weak,cardaliaguet2015second,cardaliaguet2015weak,cardaliaguet2014mean,cardaliaguet2012long,cardaliaguet2013long2,gomes2016time,gomes2015time,gomes2015time2}.
In all of these references, the equilibrium condition is determined solely through the distribution of the state variable, rather than that of the control.
That is, each player faces a cost determined by the distribution of positions, but not decisions, of other players.
For economic production models, by contrast, players must optimize against a cost determined by the distribution of \emph{controls}, since the market price is determined by aggregating all the prices (or quantities) set by individual firms.
A mathematical framework which takes this assumption into account has been called both ``extended mean field games" \cite{gomes2014extended,gomes2016extended} and ``mean field games of controls" \cite{cardaliaguet2016mfgcontrols}.
However, other than the results of \cite{gomes2014extended,gomes2016extended,cardaliaguet2016mfgcontrols}, there appear to be few existence and uniqueness theorems for PDE models of this type.
One of the main difficulties appears to be that the coupling is inherently nonlocal, a feature which is manifest in \eqref{main system} through the integral term $\int_0^L u_x m \dd x$.

Inspired by \cite{graber2016existence}, our goal in this article is to prove the existence and uniqueness of solutions to \eqref{main system}.
Because of the change in boundary conditions, many of the arguments becomes considerably simpler and stronger results are possible.
Let us now outline our main results.
We show in Section \ref{sec:analysis of the system} that there exists a unique classical solution of System \eqref{main system}.
Note that, whereas in \cite{graber2016existence}, uniqueness was only proved for small values of $\epsilon : = 2c/(1-c)$  (cf.~the interpretation in the following subsection), here we improve that result by showing that solutions are unique for all values of $\epsilon$ (including in the case of Dirichlet boundary conditions).
We show in Section \ref{sec:optlctrl} that \eqref{main system} has an interpretation as a system of optimality for a convex minimization problem.
Although this feature has been noticed and exploited for mean field games with congestion penalization (see \cite{benamou2016variational} for an overview),
here we show that it is also true for certain extended mean field games (cf.~\cite{graber2016linear}).
Finally, in Section \ref{sec:first-order} we give an existence result for the first order case where $\sigma = 0$, using a ``vanishing viscosity" argument by collecting a priori estimates from Sections \ref{sec:analysis of the system} and \ref{sec:optlctrl}.

\subsection{Explanation of the model} \label{sec:specification}

We summarize the interpretation of \eqref{main system} as follows.
Let $t$ be time and $x$ be the producer's capacity.
We assume there is a large set of producers and represent it as a continuum. 

The first equation in \eqref{main system} is the Hamilton-Jacobi-Bellman (HJB) equation for the maximization of profit.
Each producer's capacity is driven by a stochastic differential equation
\begin{equation}
dX(s) = -q(s)ds + \sigma \ dW(s),
\end{equation}
where $q$ is determined by the price $p$ through a linear demand schedule
\begin{equation}
\label{demand function}
q = D(p,\bar{p}) = \frac{1}{1+\epsilon} - p + \frac{\epsilon}{1+\epsilon}\bar{p}, \quad \eta > 0.
\end{equation}
The presence of noise expresses the the short term unpredictable fluctuations of the demand \cite{chan2015bertrand}.
In \eqref{demand function} $\bar{p}$ represents the market price, that is, the average price offered by all producers; and $\epsilon$ is the product substitutability, with $\epsilon = 0$ corresponding to independent goods and $\epsilon = +\infty$ implying perfect substitutability.
Thus each producer competes with all the others by responding to the market price. 

We define the value function
\begin{equation} \label{optimal control problem}
u(t,x) := \sup_{p} \bb{E} \left\{
\int_t^T e^{-r(s-t)}p(s)q(s) ds + e^{-r(T-t)} u_T(X(T)) \ | \  X(t) = x \right\}
\end{equation}
where $q(s)$ is given in terms of $p(s)$ by \eqref{demand function}.
The optimization problem \eqref{optimal control problem} has the corresponding Hamilton-Jacobi-Bellman equation
\begin{equation}
\label{bellman}
u_t + \frac{1}{2}\sigma^2 u_{xx} - ru + \max_{p } \left[\left(\frac{1}{1+\epsilon} - p + \frac{\epsilon}{1+\epsilon}\bar{p}(t)\right)(p-u_x)\right] = 0.
\end{equation}
The optimal $p^*(t,x)$ satisfies the first order condition
\begin{equation} \label{def equilibrium price}
p^*(t,x) = \frac{1}{2}\left(\frac{1}{1+\epsilon} + \frac{\epsilon}{1+\epsilon}\bar{p}(t) + u_x(t,x) \right),
\end{equation}
and we take $q^*(t,x)$ to be the corresponding demand
\begin{equation} \label{optimal control}
q^*(t,x) = \frac{1}{2}\left(\frac{1}{1+\epsilon} + \frac{\epsilon}{1+\epsilon}\bar{p}(t) - u_x(t,x) \right).
\end{equation}
Therefore \eqref{bellman} becomes
\begin{equation}
\label{bellman2}
u_t + \frac{1}{2}\sigma^2 u_{xx} - ru + \frac{1}{4} \left(\frac{1}{1+\epsilon} + \frac{\epsilon}{1+\epsilon}\bar{p}(t)-u_x\right)^2 = 0.
\end{equation}

On the other hand, the density of producers $m(t,x)$ is transported by the optimal control \eqref{optimal control}; it is governed by the Fokker-Planck equation
\begin{equation}
\label{fokker-planck}
m_t - (\frac{1}{2}\sigma^2 m)_{xx} - \frac{1}{2} \left(\left(\frac{1}{1+\epsilon} + \frac{\epsilon}{1+\epsilon}\bar{p}(t)-u_x\right)m\right)_x = 0.
\end{equation}
The coupling takes place through a market clearing condition.
With $p^*(t,x)$ the Nash equilibrium price we must have
\begin{equation} \label{def average price}
\bar{p}(t) = \int_0^L p^*(t,x)m(t,x)\dd x,
\end{equation}
which, thanks to \eqref{def equilibrium price}, can be rewritten
\begin{equation}
\label{average price}
\bar{p}(t) = \frac{1}{2+\epsilon} + \frac{1+\epsilon}{2+\epsilon}\int_0^L u_x(t,x)m(t,x)\dd x.
\end{equation}
We recover System \ref{main system} by setting
\begin{equation}
\label{parameters}
b = \frac{2}{2+\epsilon}, \ c = \frac{\epsilon}{2+\epsilon}.
\end{equation}
\textbf{Boundary conditions.} 
We assume that the maximum capacity of all producers does not exceed $L> 0$.
We consider a situation where players are able to renew their stock after exhaustion, so that players stay all the time with a non empty stock. For the sake of simplicity, we do not consider the implications of stock renewal on the cost structure.
This situation entails a reflection boundary condition at $x=0$ instead of an absorbing  boundary condition. Therefore, we consider Neumann boundary conditions at $x=0$ and $x=L$.

\subsection{Notation and assumptions}

Throughout this article we define $Q_T := (0,T) \times (0,L)$ to be the domain, $S_T := ([0,T] \times \{0,L\}) \cup (\{T\} \times [0,L])$ to be the parabolic boundary, and at times $\Gamma_T := ([0,T] \times \{0\}) \cup (\{T\} \times [0,L])$ to be the parabolic half-boundary.
For any domain $X$ in $\bb{R}$ or $\bb{R}^2$ we define $L^p(X)$, $p \in [1,+\infty]$ to be the Lebesgue space of $p$-integrable functions on $X$; $C^0(X)$ to be the space of all continuous functions on $X$; $C^{\alpha}(X)$, $0 < \alpha < 1$ to be the space of all H\"older continuous functions with exponent $\alpha$ on $X$; and $C^{n+\alpha}(X)$ to be the set of all functions whose $n$ derivatives are all in $C^\alpha(X)$.
For a subset $X \subset \overline{Q_T}$ we also define $C^{1,2}(X)$ to be the set of all functions on $X$ which are locally continuously differentiable in $t$ and twice locally continuously differentiable in $x$.
By $C^{\alpha/2,\alpha}(X)$ we denote the set of all functions which are locally H\"older continuous in time with exponent $\alpha/2$ and in space with exponent $\alpha$.

We will denote by $C$ a {\em generic} constant, which depends only on the data (namely $u_T,m_0,L,T,\sigma,r$ and $\epsilon$).
Its precise value may change from line to line.

Throughout we take the following assumptions on the data :
\begin{enumerate}
\item $u_{T}$ and $m_{0}$ are function in $C^{2+\gamma}([0,L])$ for some $\gamma >0$.
\item $u_{T}$ and $m_{0}$ satisfy compatible boundary conditions  : $u'_{T}(0)=u'_{T}(L)=0$ and $m_{0}(0)=m'_{0}(0)=m_{0}(L)=m'_{0}(L)=0$.
\item $m_{0}$ is probability density.
\item $u_{T}\geq 0$.
\end{enumerate}
%

\section{Analysis of the system} \label{sec:analysis of the system}
In this section we give a proof of existence and uniqueness for system \eqref{main system}. Note that most results of  this section are an adaptation of those of \cite[section 2]{graber2016existence}.  However, unlike the case addressed in \cite{graber2016existence}, we provide uniform bounds on $u$ and $u_{x}$ which do not depend on $\sigma$.
We start by providing some a priori bounds on solutions to \eqref{main system}, then we prove existence and uniqueness using the Leray-Schauder fixed point theorem. 

Let us start with some basic properties of the solutions.

\begin{proposition}\label{prop21}
Let $(u,m)$ be a pair of smooth solutions to \eqref{main system}. Then, for all $t\in [0,T]$, $m(t)$ is a probability density, and
\begin{equation}\label{estim-1-prop21}
u(t,x)\geq 0\quad \forall t\in[0,T], \forall x\in[0,L]. 
\end{equation}
Moreover, for some constant $C>0$ depending on the data, we have
\begin{equation}\label{estim-2-prop21}
\int_{0}^{T}\int_{0}^{L}m u_{x}^{2} \leq C.
\end{equation}
\end{proposition}
\begin{proof}
Using \eqref{main system}(ii) and \eqref{main system}(v), one easily checks that $m(t)$ is a probability density for all $t\in[0,T]$. Moreover, the arguments used to prove \eqref{estim-1-prop21} 
and \eqref{estim-2-prop21} 
in \cite{graber2016existence} hold also for the system \eqref{main system}.
\end{proof}

\begin{lemma}\label{lemma22}
Let $(u,m)$ be a pair of smooth solution to \eqref{main system}, then
\begin{equation} \label{estim-1-lemma22}
\|u \|_{\infty} + \|u_{x} \|_{\infty} \leq C,
\end{equation}
where the constant $C>0$ does not depend on $\sigma$. In particular we have that 
\begin{equation}\label{estim-1-nonlocal-lemma}
\forall t \in [0,T], \qquad\left\vert \int_{0}^{L} u_{x}(t,x)m(t,x)\dd x   \right\vert \leq C,
\end{equation}
where $C>0$ does not depend on $\sigma$.
\end{lemma}
\begin{proof}
As in \cite[Lemma 2.3, Lemma 2.7]{graber2016existence}, the result is a consequence of using the maximum principle for suitable functions.  We give a proof highlighting the fact that $C$ does not depend on $\sigma$. Set 
$
f(t):=b+c\int_{0}^{L}u_{x}(t,y)m(t,y) \dd y,
$
so that
$$
-u_{t}-\frac{\sigma^{2}}{2} u_{xx}+r u \leq \frac{1}{2}\left( f^{2}(t)+u_{x}^{2} \right).
$$
Owing to Proposition \ref{prop21}, $f \in L^{2}(0,T)$. Moreover, if  
$$w:= \exp\left\{ \sigma^{-2}\left( u+\frac{1}{2}\int_{0}^{t} f(s)^{2}\dd s \right)   \right\}-1,$$
then we have 
$$
-w_{t}-\frac{\sigma^{2}}{2} w_{xx} \leq 0.
$$
In particular $w$ satisfies the maximum principle, and $w\leq \mu$ everywhere, where
$$
\mu = \max_{0\leq x\leq L}   \exp\left\{ \sigma^{-2}\left( u_{T}+\frac{1}{2}\int_{0}^{T} f(s)^{2}\dd s \right)   \right\}-1.
$$
Whence, $0\leq u\leq \sigma^{2}\ln(1+\mu)$, so that
$$
\| u \|_{\infty} \leq \|u_{T}\|_{\infty}+\frac{1}{2}\int_{0}^{T} f(s)^{2}\dd s.
$$
On the other hand, we have that
$$
\max_{\Gamma_{T}} |u_{x}| \leq \| u'_{T} \|_{\infty}, \quad \Gamma_{T}:= ([0,T]\times \{ 0, L\}) \cup (\{ T\}\times [0,L]),
$$
so by  using the maximum principle for the function $w(t,x) =u_{x}(t,x)e^{-rt}$,
we infer that
$$
\| u_{x}\|_{\infty} \leq e^{r T}  \| u'_{T} \|_{\infty}.
$$
\end{proof}
\begin{remark}
Unlike in \cite{graber2016existence}, where more sophisticated estimates are performed, the estimation of the nonlocal term $\int_{0}^{L}u_{x}(t,x)m(t,x)\dd x$ follows easily in this case, owing to \eqref{estim-1-lemma22} and the fact that $m$ is a probability density.
\end{remark}

\begin{proposition}\label{prop25}
There exists a constant $C>0$ depending on $\sigma$ and data such that, if $(u,m)$ is a smooth solution to \eqref{main system}, then for some $0<\alpha<1$,
\begin{equation}\label{estim1-prop25}
\| u\|_{C^{1+\alpha/2,2+\alpha}(\overline{Q_{T}})}+\| m\|_{C^{1+\alpha/2,2+\alpha}(\overline{Q_{T}})} \leq C.
\end{equation}
\end{proposition}
\begin{proof}
See \cite[Proposition 2.8]{graber2016existence}.
\end{proof}

We now prove the main result of this section. 

\begin{theorem} \label{thm:existsunique}
There exists a unique classical solution to \eqref{main system}.
\end{theorem}
\begin{proof}
The proof of existence is the same as in \cite[Theorem 3.1]{graber2016existence} and relies on Leray-Schauder fixed point theorem. Let $(u_{1},m_{1})$ and $(u_{2},m_{2})$ be two solutions of \eqref{main system}, and set $u=u_{1}-u_{2}$ and $m=m_{1}-m_{2}$. 
Define
$$
G_{i} := \frac{1}{2}\left( b+c\int_{0}^{L}u_{i,x}(t,y)m_{i}(t,y)\dd y-u_{i,x} \right).
$$
Note that $G_i$ can be written
$$
G_i = \frac{1}{2}\left( \frac{b}{1-c}-\frac{2c}{1-c}\bar G_i-u_{i,x} \right),\quad \mbox{ where } \quad \bar G_i := \int_{0}^{L}G_i(t,y)m_{i}(t,y)\dd y.
$$
Integration by parts yields
\begin{equation} \label{eq:uniqueness1}
\left[e^{-rt}\int_0^L u(t,x)m(t,x) \dd x\right]_0^T
= \int_0^T e^{-rt} \int_0^L (G^2_2 - G^2_1 - G_1 u_{x})m_1 + (G^2_1 - G^2_2 + G_2 u_{x})m_2 \dd x \dd t.
\end{equation}
The left-hand side of \eqref{eq:uniqueness1} is zero.
As for the right-hand side, we check that
\begin{equation*}
G^2_2 - G^2_1 - G_1 u_{x} = (G_2 - G_1)^2 + \frac{2c}{1-c}G_1(\bar G_1 - \bar G_2)
\end{equation*}
and, similarly,
\begin{equation*}
G^2_1 - G^2_2 + G_2 u_{x} = (G_2 - G_1)^2 -\frac{2c}{1-c}G_2(\bar G_1 - \bar G_2).
\end{equation*}
Then \eqref{eq:uniqueness1} becomes
\begin{equation} \label{eq:uniqueness2}
0
= \int_0^T e^{-rt} \int_0^L (G_1-G_2)^2(m_1+m_2) \dd x \dd t
+ \frac{2c}{1-c}\int_0^T e^{-rt} (\bar G_1 - \bar G_2)^2 \dd t.
\end{equation}
It follows that $\bar G_1 \equiv \bar G_2$.
Then by uniqueness for parabolic equations with quadratic Hamiltonians, it follows that $u_1 \equiv u_2$.
From uniqueness for the Fokker-Planck equation it follows that $m_1 \equiv m_2$.
\end{proof}

\subsection{Uniqueness revisited for the model of Chan and Sircar}

The authors of \cite{chan2015bertrand} originally introduced the following model:
\begin{equation}
\label{old system}
\left\{
\begin{array}{rcc}
(i) & u_t + \frac{1}{2}\sigma^2 u_{xx} - ru +  G^2(t,u_x,[mu_x]) = 0, & 0 < t < T, \ 0 < x < L\\
(ii) & m_t - \frac{1}{2}\sigma^2 m_{xx} - \left(G(t,u_x,[mu_x])m\right)_x = 0, & 0 < t < T, \ 0 < x < L\\
(iii) & m(0,x) = m_0(x), u(T,x) = u_T(x), & 0 \leq x \leq L\\
(iv) & u(t,0) = m(t,0) = 0, ~~ u_x(t,L) = 0, & 0 \leq t \leq T \\ 
(v) & \frac{1}{2}\sigma^2 m_x(t,L) + G(t,u_x(t,L),[mu_x])m(t,L) = 0, & 0 \leq t \leq T
\end{array}\right.
\end{equation}
where
\begin{align}
\label{H and G} 
G(t,u_x,[mu_x]) &= \frac{1}{2}\left(\frac{2}{2+\epsilon \eta(t)} + \frac{\epsilon}{2+\epsilon \eta(t)}\int_0^L u_\xi(t,\xi) m(t,\xi) d\xi - u_x \right),
\\
\notag 
\eta(t) &:= \int_0^L m(t,\xi)d\xi
\end{align}
The main difference between \eqref{main system} and \eqref{old system} is that in \eqref{old system} there are Dirichlet boundary conditions on the left-hand side $x=0$, which also means that $m$ is no longer a density, but might have decreasing mass.
In \cite{graber2016existence}, existence and uniqueness of classical solutions for \eqref{old system} is obtained.
However, uniqueness was only proved for small parameters $\epsilon$.
Here we improve this result by using the idea of the proof of Theorem \ref{thm:existsunique}.
(The proof is in fact much simpler than in \cite{graber2016existence}.)

\begin{theorem} \label{thm:olduniqueness}
	There exists a unique classical solution of the system \eqref{old system}.
\end{theorem}

\begin{proof}
	Existence was given in \cite{graber2016existence}.
	For uniqueness, let $(u_1,m_1),(u_2,m_2)$ be two solutions, and define $u = u_1-u_2,m= m_1-m_2$, and
	\begin{align*}
	G_i &= \frac{1}{2}\left(\frac{2}{2+\epsilon \eta_i(t)} + \frac{\epsilon}{2+\epsilon \eta_i(t)}\int_0^L u_{i,\xi}(t,\xi) m_i(t,\xi) d\xi - u_{i,x} \right),
	\\
	\notag 
	\eta_i(t) &:= \int_0^L m_i(t,\xi)d\xi.
	\end{align*}
	Note that $G_i$ can also be written
	$$
	G_i = \frac{1}{2}(1 - \epsilon \bar G_i - u_{i,x}), \quad \mbox{ where } \quad \bar G_i := \int_{0}^{L}G_i(t,y)m_{i}(t,y)\dd y.
	$$
	Then integrating by parts as in the proof of Theorem \ref{thm:existsunique}, we obtain
	\begin{equation} \label{eq:uniqueness3}
	0
	= \int_0^T e^{-rt} \int_0^L (G_1-G_2)^2(m_1+m_2) \dd x \dd t
	+ \epsilon\int_0^T e^{-rt} (\bar G_1 - \bar G_2)^2 \dd t.
	\end{equation}
	We conclude as before.
\end{proof}

\section{Optimal control of Fokker-Planck equation} \label{sec:optlctrl}

The purpose of this section is to prove that \eqref{main system} is a system of optimality for a convex minimization problem.
It was first noticed in the seminal paper by Lasry and Lions \cite{lasry07} that systems of the form \eqref{eq:mfg classical} have a formal interpretation in terms of optimal control.
Since then this property has been made rigorous and exploited to obtain well-posedness in first-order \cite{cardaliaguet2015weak,cardaliaguet2014mean,cardaliaguet2015first} and degenerate cases \cite{cardaliaguet2015second}; see \cite{benamou2016variational} for a nice discussion.
However, all of these references consider the case of congestion penalization, which results in an a priori summability estimate on the density.
There is no such penalization in \eqref{main system}.
Hence, the optimality arguments used in \cite{cardaliaguet2015weak}, for example, appear insufficient in the present case to prove existence and uniqueness of solutions to the first order system.
Furthermore, it is very difficult in the present context to formulate the dual problem, which in the aforementioned works was an essential ingredient in proving existence of an adjoint state.
Nevertheless, aside from its intrinsic interest, we will see in Section \ref{sec:first-order} that optimality gives us at least enough to pass to the limit as $\sigma \to 0$.

We make the substitution $\bar{b} = \dfrac{b}{1-c}, \bar{c} = \dfrac{c}{1-c}$ (so according to \eqref{parameters} we get $\bar{b}=1$ and $\bar{c} = \epsilon/2$).
Consider the optimization problem of minimizing the objective functional
\begin{multline} \label{control of FP}
J(m,q) = \int_0^T \int_0^L e^{-rt} \left(q^2(t,x)  - \bar{b}q(t,x)\right)m(t,x)\dd x \dd t
\\
+ \bar{c}\int_0^T e^{-rt}\left(\int_0^L  q(t,y)m(t,y)\dd y\right)^2\dd t - \int_0^L e^{-rT} u_T(x)m(T,x)\dd x
\end{multline}
for $(m,q)$ in the class $\s{K}$, defined as follows.
Let  $m \in L^1([0,T] \times [0,L])$ be non-negative, let $q \in L^2([0,T] \times [0,L])$, and assume that $m$ is a weak solution to the Fokker-Planck equation
\begin{equation}
\label{fokker planck condition}
m_t -  \frac{\sigma^2}{2}m_{xx} - (qm)_x = 0, \ \ m(0) = m_0,
\end{equation}
equipped with Neumann boundary conditions, where weak solutions are defined as in \cite{porretta2015weak}:
\begin{itemize}
	\item the integrability condition $mq^2 \in L^1([0,T] \times [0,L])$ holds, and
	\item \eqref{fokker planck condition} holds in the sense of distributions--namely, for all $\phi \in C^{\infty}_c([0,T) \times [0,L])$ such that $\phi_x(t,0) = \phi_x(t,L) = 0$ for each $t \in (0,T)$, we have
	$$
	\int_0^T \int_0^L (-\phi_t - \frac{\sigma^2}{2}\phi_{xx} + q\phi_x)m \ \dd x \dd t = \int_0^L \phi(0)m_0 \dd x.
	$$
\end{itemize}
Then we say that $(m,q) \in \s{K}$.
We refer the reader to \cite{porretta2015weak} for properties of weak solutions of \eqref{fokker planck condition}, namely that they are unique and that they coincide with renormalized solutions and for this reason have several useful properties.
One property which will be of particular interest to us is the following lemma:
\begin{lemma}[Proposition 3.10 in \cite{porretta2015weak}] \label{lem:positivity}
	Let $(m,q) \in \s{K}$, i.e.~let $m$ be a weak solution of the Fokker-Planck equation \eqref{fokker planck 
	condition}.
	Then $\|m(t)\|_{L^1([0,L])} = \|m_0\|_{L^1([0,L])}$ for all $t \in [0,T]$.
	Moreover, if $\log m_0 \in L^1([0,L])$, then for any 
	\begin{equation}
	\|\log m(t)\|_{L^1([0,L])} \leq  C(\|\log m_0\|_{L^1([0,L])} + 1) \ \ \forall t \in [0,T],
	\end{equation}
	where $C$ depends on $\|q\|_{L^2}$ and $\|m_0\|_{L^1}$.
	In particular, if $\log m_0 \in L^1([0,L])$ and $(m,q)$ in $\s{K}$, then $m > 0$ a.e.
\end{lemma}

\begin{proposition} \label{prop:minimizers}
	Let $(u,m)$ be a solution of \eqref{main system}.
	Set 
	$$q = \frac{1}{2}\left(b + c\int_0^L u_x(t,y) m(t,y) \dd y - u_x\right).$$
	Then $(m,q)$ is a minimizer for problem \eqref{control of FP}, that is, $J(m,q) \leq J(\tilde m,\tilde q)$ for all $
	(\tilde m,\tilde q)$ satisfying \eqref{fokker planck condition}.
	Moreover, if $\log m_0 \in L^1([0,L])$ then the maximizer is unique.
\end{proposition}

\begin{proof}
	It is useful to keep in mind that the proof is based on the convexity of $J$ following a change of variables.
	By abuse of notation we might write
	\begin{multline*}
	J(m,w) = \int_0^T \int_0^L e^{-rt} \left(\frac{w^2(t,x)}{m(t,x)}  - \bar{b}w(t,x)\right)\dd x \dd t
	\\
	+ \bar{c}\int_0^T e^{-rt}\left(\int_0^L  w(t,y)\dd y\right)^2\dd t - \int_0^L e^{-rT} u_T(x)m(T,x)\dd x,
	\end{multline*}
	cf.~the change of variables used in \cite{benamou2000computational} and several works which cite that paper.
	However, in this context we prefer a direct proof.
	
	Using the algebraic identity
	$$
	\tilde q^2\tilde m - q^2 m = 2q(\tilde q \tilde m - qm) - q^2(\tilde m - m) + \tilde m(\tilde q - q)^2,
	$$
	we have
	\begin{multline} \label{eq:optimality1}
	J(\tilde m,\tilde q) - J(m,q) = \bar{c}\int_0^T e^{-rt}\left( \int_0^L\tilde q \tilde m - qm \ \dd y\right)^2 \dd t
	- \int_0^L e^{-rT} u_T(x)(\tilde m -m)(T,x)\dd x
	\\
	+ 2\bar{c}\int_0^T e^{-rt}\left( \int_0^L\tilde q \tilde m - qm \ \dd y\right)\left( \int_0^Lqm \ \dd y\right)\dd t
	\\ +\int_0^T \int_0^L e^{-rt} \left(\bar{b}(qm - \tilde q \tilde m)  + 2q(\tilde q \tilde m - qm) - q^2(\tilde m - m) 
	+ \tilde m(\tilde q - q)^2\right)\dd x \dd t.
	\end{multline}
	Now using the fact that $u$ is a smooth solution of
	\begin{equation}
	\label{eq:adjoint}
	u_t + \frac{\sigma^2}{2}u_{xx} - ru + q^2 = 0, \ u(T) = 0, \ u_x|_{0,L} = 0
	\end{equation}
	and since
	$$
	(\tilde m - m)_t - \frac{\sigma^2}{2}(\tilde m - m)_{xx} - (\tilde q \tilde m - qm)_x = 0, \ \ (\tilde m - m)(0) = 0
	$$
	in the sense of distributions, it follows that
	\begin{multline*}
	\int_0^T \int_0^L e^{-rt} q^2(\tilde m - m)\dd x \dd t
	+ \int_0^L e^{-rT} u_T(x)(\tilde m -m)(T,x)\dd x
	\\
	= -\int_0^T \int_0^L e^{-rt}(\tilde q \tilde m - qm)u_x \ \dd x \dd t.
	\end{multline*}
	Putting this into \eqref{eq:optimality1} and rearranging, we have
	\begin{multline} \label{eq:optimality2}
	J(\tilde m,\tilde q) - J(m,q) = \int_0^T \int_0^L e^{-rt} (qm - \tilde q \tilde m) \left(\bar{b}  - 2q - 2\bar{c}\int_0^L  
	qm \ \dd y - u_x\right)\dd x \dd t
	\\
	+ \int_0^T \int_0^L e^{-rt} \tilde m(\tilde q - q)^2 \dd x \dd t
	+ \bar{c}\int_0^T e^{-rt}\left( \int_0^L\tilde q \tilde m - qm \ \dd x\right)^2 \dd t.
	\end{multline}
	To conclude that $J(\tilde m,\tilde q) \geq J(m,q)$, it suffices to prove that
	\begin{equation}
	\label{eq:first-order}
	\bar{b}  - 2q - 2\bar{c}\int_0^Lqm \ \dd y - u_x = 0.
	\end{equation}
	Recall the definition
	$$q = \frac{1}{2}\left(b + c\int_0^L u_x(t,y) m(t,y) \dd y - u_x\right).$$
	Integrate both sides against $m$ and rearrange, using the definition of the constants $\bar b,\bar c$ to get
	$$
	\int u_xm ~\dd y = \bar b - 2(\bar c + 1)\int qm ~\dd y.
	$$
	Plugging this into the definition of $q$ proves \eqref{eq:first-order}.
	Thus $(m,q)$ is a minimizer.
	
	On the other hand, suppose $\log m_0 \in L^1([0,L])$ and that $(\tilde m,\tilde q)$ is another minimizer.
	Then \eqref{eq:optimality2} implies that 
	\begin{equation} \label{eq:optimality3}
	\int_0^T \int_0^L e^{-rt} \tilde m(\tilde q - q)^2 \dd x \dd t
	+ \bar{c}\int_0^T e^{-rt}\left( \int_0^L\tilde q \tilde m - qm \ \dd x\right)^2 \dd t = 0.
	\end{equation}
	Now by Lemma \ref{lem:positivity}, we have $\tilde m > 0$ a.e.
	Therefore \eqref{eq:optimality3} implies $\tilde q = q$.
	By uniqueness for the Fokker-Planck equation, we conclude that $\tilde m = m$ as well.
	The proof is complete.
\end{proof}

\begin{remark}
	A similar argument shows that System \eqref{old system}, with Dirichlet boundary conditions on the left-hand side, is also a system of optimality for the same minimization problem, except this time with Dirichlet boundary conditions (on the left-hand side) imposed on the Fokker-Planck equation.
	We omit the details.
\end{remark}

\section{First-order case} \label{sec:first-order}

In this section we use a vanishing viscosity method to prove that \eqref{main system} has a solution even when we plug in $\sigma = 0$.
We need to collect some estimates which are uniform in $\sigma$ as $\sigma \to 0$.
From now on we will assume $0 < \sigma \leq 1$, and whenever a constant $C$ appears it does not depend on $\sigma$.

\begin{lemma} \label{lem:ut L2}
	$\|u_t\|_2 \leq C$.
\end{lemma}

\begin{proof}
	We first prove that $\sigma^2 \|u_{xx}\|_2 \leq C$.
	For this, multiply
	\begin{equation}
	u_{xt} - ru_x + \frac{\sigma^2}{2}u_{xxx} - Gu_{xx} = 0
	\end{equation}
	by $u_{x}$ and integrate by parts.
	We get, after using Young's inequality and  \eqref{estim-1-lemma22},
	$$
	\sigma^4 \int_0^T \int_0^L u_{xx}^2 \dd x \dd t \leq 4\int_0^T \int_0^L (Gu_{x})^2 \dd x \dd t 
	+ 2\sigma^{2}\int_0^L u_T'(x)^2 \dd x \leq C,
	$$
	as desired.
	
	Then the claim follows from \eqref{main system}(i) and Lemma \ref{lemma22}.
\end{proof}

\begin{lemma} \label{lem:u is holder}
	$\|u\|_{C^{1/3}} \leq C$.
\end{lemma}

\begin{proof}
	Since $\|u_x\|_\infty \leq C$ it is enough to show that $u$ is 1/3-H\"older continuous in time.
	Let $t_1 < t_2$ in $[0,T]$ be given.
	Set $\eta > 0$ to be chosen later.
	We have, by H\"older's inequality,
	\begin{multline}
	|u(t_1,x)-u(t_2,x)| \leq C\eta + \frac{1}{\eta}\int_{x-\eta}^{x+\eta} |u(t_1,\xi)-u(t_2,\xi)|\dd\xi
	\\
	\leq C\eta + \frac{1}{\eta}\int_{x-\eta}^{x+\eta} \int_{t_1}^{t_2} |u_t(s,\xi)|\dd s \dd\xi
	\\
	\leq C\eta + \frac{1}{\eta}\|u_t\|_2 \sqrt{2\eta |t_2-t_1|}
	\leq C\eta + C|t_2-t_1|^{1/2}\eta^{-1/2}.
	\end{multline}
	Setting $\eta = |t_2-t_1|^{1/3}$ proves the claim.
\end{proof}

To prove compactness estimates for $m$, we will first use the fact that it is the minimizer for an optimization problem.
Let us reintroduce the optimization problem from Section \ref{sec:optlctrl} with $\sigma \geq 0$ as a variable.
We first define the convex functional
\begin{equation} \label{eq:Psi}
\Psi(m,w) := \left\{\begin{array}{cl}
\frac{|w|^2}{m} & \text{if}~m \neq 0,\\
0 &\text{if}~w= 0, m = 0,\\
+\infty &\text{if}~w \neq 0, m=0.
\end{array}\right.
\end{equation}
Now we rewrite the functional $J$, with a slight abuse of notation, as
\begin{multline} \label{control of FP sigma}
J(m,w) = \int_0^T \int_0^L e^{-rt} \left(\Psi(m(t,x),w(t,x))  - \bar{b}w(t,x)\right)\dd x \dd t
\\
+ \bar{c}\int_0^T e^{-rt}\left(\int_0^L  w(t,y)\dd y\right)^2\dd t - \int_0^L e^{-rT} u_T(x)m(T,x)\dd x,
\end{multline}
and consider the problem of minimizing over the class $\s{K}_\sigma$, defined here as the set of all pairs $(m,w) \in L^1((0,T) \times (0,L))_+ \times L^1((0,T) \times (0,L);\bb{R}^d)$ such that
\begin{equation}
\label{eq:mw-constraint}
m_t - \frac{\sigma^2}{2} m_{xx} -  w_x = 0, \ m(0) = m_0
\end{equation}
in the sense of distributions.
By Proposition \ref{prop:minimizers}, for every $\sigma > 0$, $J$ has a minimizer in $\s{K}_\sigma$ given by $(m,w) = (m,G m)$ where  $(u,m)$ is the solution of System \eqref{main system}.
Since $(m,w)$ is a minimizer, we can derive a priori bounds which imply, in particular, that $m(t)$ is H\"older continuous in the Kantorovich-Rubinstein distance on the space of probability measures, with norm bounded uniformly in $\sigma$.
We recall that the Kantorovich-Rubinstein metric on $\s{P}(\Omega)$, the space of Borel probability measures on $\Omega$, is defined by
$$
{\mathbf d} _1(\mu,\nu) = \inf_{\pi \in \Pi(\mu,\nu)}\int_{\Omega \times \Omega} |x-y|\dd\pi(x,y),
$$
where $\Pi(\mu,\nu)$ is the set of all probability measures on $\Omega \times \Omega$ whose first marginal is $\mu$ and whose second marginal is $\nu$.
Here we consider $\Omega = (0,L)$.

\begin{lemma} \label{lem:w2m L1}
	There exists a constant $C$ independent of $\sigma$ such that 
	$$
	\||w|^2/m\|_{L^1((0,T) \times (0,L))} \leq C. 
	$$
	As a corollary, $m$ is $1/2$-H\"older continuous from $[0,T]$ into $\s{P}((0,L))$, and there exists a constant (again denoted $C$) independent of $\sigma$ such that
	\begin{equation} \label{eq:holder in P}
	{\mathbf d}_1(m(t_1),m(t_2)) \leq C|t_1-t_2|^{1/2}.
	\end{equation}
\end{lemma}

\begin{proof}
	To see that $\||w|^2/m\|_{L^1((0,T) \times (0,L))} \leq C$, use $(m_0,0) \in \s{K}$ as a comparison.
	By the fact that $J(m,w) \leq J(m_0,0)$ we have
	\begin{multline*}
	\int_0^T \int_0^L e^{-rt}\frac{|w|^2}{2m} \dd x \dd t
	+ \bar{c} \int_0^T e^{-rt} \left(\int_0^L w \dd x\right)^2 \dd t
	\\
	\leq \int_0^L e^{-rT} u_T(m(T)-m_0) \dd x
	+ \frac{\bar{b}}{2}\int_0^T \int_0^L e^{-rt} m \dd x \dd t
	\leq C.
	\end{multline*}
	The H\"older estimate \eqref{eq:holder in P} follows from \cite[Lemma 4.1]{cardaliaguet2015second}.
\end{proof}

We also have compactness in $L^1$, which comes from the following lemma.

\begin{lemma} \label{lem:m unif intble}
	For every $K \geq 0$, we have
	\begin{equation} \label{eq:m uniformlyintegrable}
	\int_{m(t) \geq 2K} m(t) \dd x \leq 2\int_0^L (m_0-K)_+ \dd x
	\end{equation}
	for all $t \in [0,T]$.
\end{lemma}

\begin{proof}
	Let $K \geq 0$ be given.
	We define the following auxiliary functions:
	\begin{equation}
	\phi_{\alpha,\delta}(s) := \left\{
	\begin{array}{ll}
	0 & \text{if}~s \leq K,\\
	\frac{1}{6}(1+\alpha)\alpha \delta^{\alpha-2}(s-K)^3 & \text{if}~K \leq s \leq K + \delta,\\
	\frac{1}{6}(1+\alpha)\alpha \delta^{\alpha+1} + \frac{1}{2}(1+\alpha)\alpha \delta^{\alpha}(s-K) + (s-K)^{1+\alpha} & \text{if}~ s \geq K + \delta,
	\end{array}
	\right.
	\end{equation}
	where $\alpha,\delta \in (0,1)$ are parameters going to zero.
	For reference we note that
	\begin{equation}
	\phi_{\alpha,\delta}'(s) = \left\{
	\begin{array}{ll}
	0 & \text{if}~s \leq K,\\
	\frac{1}{2}(1+\alpha)\alpha \delta^{\alpha-2}(s-K)^2 & \text{if}~K \leq s \leq K + \delta,\\
	\frac{1}{2}(1+\alpha)\alpha \delta^{\alpha} + (1+\alpha)(s-K)^{\alpha} & \text{if}~ s \geq K + \delta,
	\end{array}
	\right.
	\end{equation}
	and
	\begin{equation}
	\phi_{\alpha,\delta}''(s) = \left\{
	\begin{array}{ll}
	0 & \text{if}~s \leq K,\\
	(1+\alpha)\alpha \delta^{\alpha-2}(s-K) & \text{if}~K \leq s \leq K + \delta,\\
	(1+\alpha)\alpha(s-K)^{\alpha-1} & \text{if}~ s \geq K + \delta.
	\end{array}
	\right.
	\end{equation}
	Observe that  $\phi_{\alpha,\delta}''$ is continuous and non-negative.
	Multiply \eqref{main system}(ii) by $\phi_{\alpha,\delta}'(m)$ and integrate by parts.
	After using Young's inequality we have
	\begin{equation}
	\int_0^L \phi_{\alpha,\delta}(m(t))\dd x \leq \int_0^L \phi_{\alpha,\delta}(m_0) \dd x
	\\
	+ \frac{\|G\|_\infty^2}{2\sigma^2}\int_0^t \int_0^L \phi_{\alpha,\delta}''(m)m^2 \dd x \dd t.
	\end{equation}
	Since $\phi_{\alpha,\delta}''(s) \leq (1+\alpha)\alpha \delta^{-2}$, after taking $\alpha \to 0$ we have
	\begin{equation}
	\int_0^L \phi_{\delta}(m(t))\dd x \leq \int_0^L \phi_{\delta}(m_0) \dd x,
	\end{equation}
	where $\phi_{\delta}(s) = (s-K)\chi_{[K+\delta,\infty)}(s)$.
	Now letting $\delta \to 0$ we see that
	\begin{equation}
	\int_0^L (m(t) - K)_+ \dd x \leq \int_0^L (m_0-K)_+ \dd x,
	\end{equation}
	where $s_+ := (s+|s|)/2$ denotes the positive part. Whence
	\begin{equation}
	\int_0^L (m_\sigma(t) - K)_+ \dd x \leq \int_0^L (m_0-K)_+ \dd x,
	\end{equation}
	which also implies \eqref{eq:m uniformlyintegrable}.
\end{proof}

We also have a compactness estimate for the function $t \mapsto \int_0^L u_x(t,y)m(t,y) \dd y$.

\begin{lemma}
	\label{lem:mx regularity}
	$\sigma^2 \left(\int_{0}^{T}\int_0^L \frac{|m_x|^2}{m + 1}\dd x \dd t\right)^{1/2} \leq C$.
\end{lemma}

\begin{proof}
	Multiply the Fokker-Planck equation by $\log(m+1)$ and integrate by parts.
	After using Young's inequality, we obtain
	\begin{multline*}
	\frac{\sigma^4}{4}\int_{0}^{T}\int_0^L \frac{|m_x|^2}{m + 1}\dd x \dd t
	\leq \sigma^2 \int_0^L \left((m_0+1)\log(m_0+1) - m_0\right)\dd x
	+ \|G\|_\infty^2 \int_{0}^{T}\int_0^L \frac{m^2}{m+1} 
	\\
	\leq \int_0^L \left((m_0+1)\log(m_0+1) - m_0\right)\dd x
	+ \|G\|_\infty^2 \int_{0}^{T}\int_0^L m \dd x \dd t 
	\leq C.
	\end{multline*}
\end{proof}

\begin{lemma} \label{lem:I(t) Holder}
	Let $\zeta \in C_c^\infty((0,L))$.
	Then $t \mapsto \int_0^L u_x(t,x)m(t,x)\zeta(x)\dd x$ is $1/2$-H\"older continuous, and in particular,
	\begin{equation}
	\left|\left[\int_0^L u_x(t,x)m(t,x)\zeta(x)\dd x\right]_{t_1}^{t_2}\right| \leq C_\zeta|t_1-t_2|^{1/2}
	\end{equation}
	where $C_\zeta$ is a constant that depends on $\zeta$ but not on $\sigma$.
\end{lemma}

\begin{proof}
	Integration by parts yields
	\begin{multline}
	\left[ e^{-rt}\int_0^L u_x(t,x)m(t,x)\zeta(x)\dd x\right]_{t_1}^{t_2}
	\\
	= -\sigma^2 \int_{t_1}^{t_2}e^{-rs}\int_0^L u_x(t,x)m_x(t,x)\zeta'(x)\dd x \dd s
	- \frac{\sigma^2}{2}  \int_{t_1}^{t_2}e^{-rs}\int_0^L u_x(t,x)m(t,x)\zeta''(x)\dd x \dd s
	\\
	- \frac{1}{2}\int_{t_{1}}^{t_{2}}\left\{\left(b+c\int_{0}^{L}u_{x}(t)m(t)\right)\int_{0}^{L}\zeta_{x}u_{x}m \dd x - \int_{0}^{L}
	\zeta_{x}u_{x}^{2}m \dd x \right\} \dd s.
	\end{multline}
	On the one hand,
	$$
	\left|\frac{\sigma^2}{2}  \int_{t_1}^{t_2}e^{-rs}\int_0^L u_x(t,x)m(t,x)\zeta''(x)\dd x \dd s\right|
	\leq \frac{\|u_x\|_\infty \|\zeta''\|_\infty}{2}|t_1-t_2| \leq C\|\zeta''\|_\infty|t_1-t_2|,
	$$
	and
	$$
	 \left| \int_{t_{1}}^{t_{2}}\left\{\left(b+c\int_{0}^{L}u_{x}(t)m(t)\right)\int_{0}^{L}\zeta_{x}u_{x}m \dd x - \int_{0}^{L}
	 \zeta_{x}u_{x}^{2}m \dd x \right\} \dd s \right| \leq C \| \zeta'\|_{\infty}\|u_{x}\|_{\infty}^{2}|t_{1}-t_{2}|.
	$$
	On the other hand, by H\"older's inequality and Lemma \ref{lem:mx regularity} we get
	\begin{multline*}
	\left|\sigma^2 \int_{t_1}^{t_2}e^{-rs}\int_0^L u_x(t,x)m_x(t,x)\zeta'(x)\dd x \dd s\right|
	\\
	\leq \|u_x\|_\infty \|\zeta'\|_\infty  \sigma^2 \left(\int_{t_1}^{t_2}\int_0^L \frac{|m_x|^2}{m + 1}\dd x \dd s
	\right)^{1/2}\left(\int_{t_1}^{t_2}\int_0^L (m + 1)\dd x \dd s\right)^{1/2}
	\\
	\leq C\|\zeta'\|_\infty(L+1)^{1/2}|t_1-t_2|^{1/2}.
	\end{multline*}
\end{proof}

\begin{corollary}
	\label{cor:I(t) equicontinuity}
	The function $t \mapsto \int_0^L u_x(t,x)m(t,x) \dd x$ is uniformly continuous with modulus of continuity independent of $\sigma$.
\end{corollary}

\begin{proof}
	Let $\delta \in (0,L)$ and fix $\zeta \in C_c^\infty((0,L))$ be such that $0 \leq \zeta \leq 1$ and $\zeta \equiv 1$ on $[\delta,L-\delta]$.
	Notice that for any $t_1,t_2 \in [0,T]$
	\begin{equation} \label{eq:I(t) equicty 1}
	\left|\left[\int_0^L u_x(t,x)m(t,x)(1-\zeta(x)) \dd x\right]_{t_1}^{t_2}\right|
	\leq \|u_x\|_\infty \int_{[0,L] \setminus [\delta,L-\delta]} [m(t_1,x) + m(t_2,x)] \dd x.
	\end{equation}
	Now by Lemma \ref{lem:m unif intble} we have
	\begin{multline} \label{eq:I(t) equicty 2}
	\int_{[0,L] \setminus [\delta,L-\delta]} m(t,x) \dd x
	\\
	\leq \int_{\{ m(t) < 2K\} \cap [0,L] \setminus [\delta,L-\delta]} m(t,x) \dd x
	+ \int_{\{ m(t) \geq 2K\}} m(t,x) \dd x
	\leq 4K\delta + 2\int_0^L (m_0-K)_+ \dd x
	\end{multline}
	for all $t \in [0,T]$.
	Combine \eqref{eq:I(t) equicty 1} and \eqref{eq:I(t) equicty 2} with Lemmas \ref{lem:I(t) Holder} and 
	\ref{lemma22} to get
	\begin{equation} \label{eq:I(t) equicty 3}
	\left|\left[\int_0^L u_x(t,x)m(t,x) \dd x\right]_{t_1}^{t_2}\right|
	\leq C_\zeta |t_1-t_2|^{1/2} + CK\delta + C\int_0^L (m_0-K)_+ \dd x \ \ \forall t_1,t_2 \in [0,T].
	\end{equation}
	Let $\eta > 0$ be given.
	Set $K$ large enough such that $C\int_0^L (m_0-K)_+ \dd x < \eta/3$, then pick $\delta$ small enough that $CK\delta < \eta/3$.
	Finally, fix $\zeta$ as described above.
	Equation \eqref{eq:I(t) equicty 3} implies that if $|t_1-t_2| < \eta^2/(9C_\zeta^2)$, we have $\left|\left[\int_0^L u_x(t,x)m(t,x) \dd x\right]_{t_1}^{t_2}\right| < \eta$.
	Thus the function $t \mapsto \int_0^L u_x(t,x)m(t,x) \dd x$ is uniformly continuous, and since none of the constants here depend on $\sigma$, the modulus of continuity is independent of $\sigma$.
\end{proof}

We are now in a position to prove an existence result for the first-order system.
\begin{theorem}
	There exists a unique pair $(u,m)$ which solves System \eqref{main system} in the following sense:
	\begin{enumerate}
		\item $u \in W^{1,2}([0,T] \times [0,L]) \cap L^\infty(0,T;W^{1,\infty}(0,L))$ is a continuous solution of the Hamilton-Jacobi equation
		\begin{equation}
		\label{eq:hj1storder}
		u_t - ru + \frac{1}{4}(f(t)-u_x)^2 = 0, \ u(T,x) = u_T(x),
		\end{equation}
		equipped with Neumann boundary conditions, in the viscosity sense;
		\item $m \in L^1\cap C([0,T];\s{P}([0,L]))$ satisfies the continuity equation
		\begin{equation}
		\label{eq:continuity}
		m_t - \frac{1}{2}((f(t)-u_x)m)_x = 0, \ m(0) = m_0,
		\end{equation} equipped with Neumann boundary conditions, in the sense of distributions; and
		\item $f(t) = b + c\int_0^L u_x(t,x)m(t,x)\dd x$ a.e.
	\end{enumerate}
\end{theorem}

\begin{proof}
	\emph{Existence:} Collecting Lemmas \ref{lemma22}, \ref{lem:ut L2} \ref{lem:u is holder}, \ref{lem:w2m L1}, \ref{lem:m unif intble}, and Corollary \ref{cor:I(t) equicontinuity}, we can construct a sequence $\sigma_n \to 0^+$ such that if $(u^n,m^n)$ is the solution corresponding to $\sigma = \sigma_n$, we have
	\begin{itemize}
		\item $u^n \to u$ uniformly, so that $u \in C([0,T] \times [0,L])$, and also weakly in $W^{1,2}([0,T] \times [0,L])$;
		\item $u_x^n \rightharpoonup u_x$ weakly$^*$ in $L^\infty$;
		\item $m^n \to m$ in $C([0,T];\s{P}([0,L]))$, so that $m(t)$ is a well-defined probability measure for every $t \in [0,T]$, $m^n \rightharpoonup m$ weakly in $L^1([0,T] \times [0,L])$, and $m^n(T) \rightharpoonup m(T)$ weakly in $L^1([0,L])$;
		\item $m^n u_x^n \rightharpoonup w$ weakly in $L^1$; and
		\item $f_n(t) := b +  c\int_0^L u_x^n(t,x)m^n(t,x)\dd x \to f(t)$ in $C([0,T])$.
	\end{itemize}
	Since $u^n \to u$ and $f_n \to f$ uniformly, by standard arguments, we have that
	\eqref{eq:hj1storder} holds
	in a viscosity sense.
	Moreover, since $u_x^n \rightharpoonup u_x$ weakly$^*$ in $L^\infty$, we also have
	\begin{equation}
	\label{eq:hj1storderinequality}
	u_t - ru + \frac{1}{4}(f(t)-u_x)^2 \leq 0
	\end{equation}
	in the sense of distributions, i.e.~for all $\phi \in C^\infty([0,T] \times [0,L])$ such that $\phi \geq 0$, we have
	\begin{multline} \label{eq:hj1storderindistrib}
	\int_0^L e^{-rT}u_T(x)\phi(T,x)\dd x
	- \int_0^L e^{-rT}u(0,x)\phi(0,x)\dd x
	\\ - \int_0^T \int_0^L e^{-rt}u(t,x)\phi_t(t,x)\dd x \dd t
	+ \frac{1}{4}\int_0^T \int_0^L (f(t)-u_x(t,x))^2 \phi(t,x) \dd x \dd t \leq 0.
	\end{multline}
	(This follows from the convexity of $u_x \mapsto u_x^2$.)
	
	Since $m^n \rightharpoonup m$ and $m^n u_x^n \rightharpoonup w$ weakly in $L^1$, it also follows that
	\begin{equation}
	\label{eq:fp1storder}
	m_t - \frac{1}{2}(f(t)m-w)_x = 0, \ m(0) = m_0
	\end{equation}
	in the sense of distributions.
	For convenience we define $\upsilon := \frac{1}{2}(f(t)m-w)$.
	Extend the definition of $(m,\upsilon)$ so that $m(t,x) = m(T,x)$ for $t \geq T$, $m(t,x) = m_0(x)$ for $t \leq 0$, and $m(t,x) = 0$ for $x \notin [0,L]$; and so that $\upsilon(t,x) = 0$ for $(t,x) \notin [0,T] \times [0,L]$.
	Now let $\xi_\delta(t,x)$ be a standard convolution kernel (i.e.~a $C^\infty$, positive function whose support is contained in a ball of radius $\delta$ and such that $\iint \xi^\delta(t,x) \dd x \dd t = 1$).
	Set $m_\delta = \xi_\delta \ast m$ and $\upsilon_\delta = \xi_\delta$.
	Then $m_\delta,\upsilon_\delta$ are smooth functions such that $\partial_t m_\delta = \partial_x \upsilon_\delta$ in $[0,T] \times [0,L]$; moreover $m_\delta$ is positive.
	Using $m_\delta$ as a test function in \eqref{eq:hj1storderindistrib} we get
	\begin{multline*} 
	\int_0^L e^{-rT}u_T(x)m_\delta(T,x)\dd x
	- \int_0^L e^{-rT}u(0,x)m_\delta(0,x)\dd x
	\\ + \int_0^T \int_0^L e^{-rt}u_x \upsilon_\delta \dd x \dd t
	+ \frac{1}{4}\int_0^T \int_0^L (f(t)-u_x)^2 m_\delta \dd x \dd t \leq 0.
	\end{multline*}
	Using the continuity of $m(t)$ in $\s{P}([0,L])$ from Lemma \ref{lem:w2m L1}, we see that 
	$$\lim_{\delta \to 0^+} \int_0^L e^{-rT}u_T(x)m_\delta(T,x)\dd x = \int_0^L e^{-rT}u_T(x)m(T,x)\dd x,$$ and 
	$\lim_{\delta \to 0^+} \int_0^L e^{-rT}u(0,x)m_\delta(0,x)\dd x = \int_0^L e^{-rT}u(0,x)m_0(x)\dd x.$
	Since $m_\delta \to m$ and $\upsilon_\delta \to \upsilon$ in $L^1$, we have
	\begin{multline*} 
	\int_0^L e^{-rT}u_T(x)m(T,x)\dd x
	- \int_0^L e^{-rT}u(0,x)m_0(x)\dd x
	\\ + \int_0^T \int_0^L e^{-rt}u_x \upsilon  \dd x \dd t
	+ \frac{1}{4}\int_0^T \int_0^L (f(t)-u_x)^2 m \dd x \dd t \leq 0,
	\end{multline*}
	or
	\begin{multline} \label{eq:ibpestimate1}
	\int_0^L e^{-rT}u_T(x)m(T,x)\dd x
	- \int_0^L e^{-rT}u(0,x)m_0(x)\dd x
	\\ + \int_0^T \int_0^L e^{-rt} \left(\frac{1}{4}mu_x^2 - \frac{1}{2}u_x w\right)  \dd x \dd t
	+ \frac{1}{4}\int_0^T \int_0^L f^2(t)m \dd t \leq 0.
	\end{multline}
	Recall the definition of $\Psi(m,w)$ from \eqref{eq:Psi}.
	From \eqref{eq:ibpestimate1} we have
	\begin{multline} \label{eq:ibpestimate2}
	\int_0^L e^{-rT}u_T(x)m(T,x)\dd x
	- \int_0^L e^{-rT}u(0,x)m_0(x)\dd x
	\\
	+ \frac{1}{4}\int_0^T \int_0^L f^2(t)m \dd t \leq 
	\frac{1}{4}\int_0^T \int_0^L e^{-rt}  \Psi(m,w) \dd x \dd t.
	\end{multline}
	On the other hand, for each $n$ we have
	\begin{multline} \label{eq:ibpequation}
	\int_0^L e^{-rT}u_T(x)m^n(T,x)\dd x
	- \int_0^L e^{-rT}u^n(0,x)m_0(x)\dd x
	\\
	+ \frac{1}{4}\int_0^T \int_0^L f_n^2(t)m^{n} \dd t 
	= \frac{1}{4}\int_0^T \int_0^L e^{-rt} m^{n}u_x^2 \dd x \dd t
	= \frac{1}{4}\int_0^T \int_0^L e^{-rt}  \Psi(m^n,m^nu_x^n) \dd x \dd t.
	\end{multline}
	Since $(m^n,m^nu_x^n) \rightharpoonup (m,w)$ weakly in $L^1 \times L^1$, it follows from weak lower semicontinuity that
	\begin{multline} \label{eq:ibpestimate3}
	\int_0^L e^{-rT}u_T(x)m(T,x)\dd x
	- \int_0^L e^{-rT}u(0,x)m_0(x)\dd x
	\\
	+ \frac{1}{4}\int_0^T \int_0^L f^2(t)m \dd t \geq 
	\frac{1}{4}\int_0^T \int_0^L e^{-rt}  \Psi(m,w) \dd x \dd t.
	\end{multline}
	From \eqref{eq:ibpestimate1}, \eqref{eq:ibpestimate2}, and \eqref{eq:ibpestimate3} it follows that
	$$
	\int_0^T \int_0^L e^{-rt}  (\Psi(m,w) + mu_x^2 - 2u_x w) \dd x \dd t = 0,
	$$
	where $\Psi(m,w) + mu_x^2 - 2u_x w$ is a non-negative function, hence zero almost everywhere.
	We deduce that $w = mu_x$ almost everywhere.
	
	Finally, by weak convergence we have 
	\begin{eqnarray}\nonumber
	f(t) = b + c \lim_{n \to \infty} \int_0^L u_x^n(t,x)m^n(t,x) \dd x &=& b + c\int_0^L w(t,x) \dd x \\ \nonumber
	&=& b + c\int_0^L  u_x(t,x)m(t,x) \dd x\quad a.e.
	\end{eqnarray}
	Which entails the existence part of the Theorem.

	\emph{Uniqueness:} The proof of uniqueness is essentially the same as for the second order case, the only 
	difference is the lack of regularity which makes the arguments much more subtle invoking results for transport   
	equations with a non-smooth vector field. Let $(u_{1}, m_1)$ and $
	(u_2,m_2)$ be two solutions of system \eqref{main system} in the sense given above, and let us set $u:=u_{1}-
	u_{2}$ and $m=m_1-m_2$. 
	We use a regularization process to get the energy estimate \eqref{eq:uniqueness2}. Then we get that 
	$u_1\equiv u_2$ and 
	$\int_0^L u_{1,x}m_1=\int_{0}^{L} u_{2,x}m_2$ in $\{ m_{1}>0\}\cup \{ m_2>0\}$, so that $m_1$ and 
	$m_2$ are both solutions to 
	$$
	m_t - \frac{1}{2}((f_{1}(t)-u_{1,x})m)_x = 0, \ \ m(0) = m_0,
	$$
	where $f_{1}(t):=b + c\int_0^L u_{1,x}(t,x)m_{1}(t,x)\dd x$ and $u_{1,x}:=(u_{1})_{x}$.  In orded to conclude that $m_1\equiv m_2$, we  invoke 
	the following Lemma:
	\begin{lemma}\label{uniqueness:continuity:equation}
	Assume that $v$ is a viscosity solution to 
	$$
	v_{t} -rv+\frac{1}{4}(f_{1}(t)-v_{x})^{2}=0, \ \ v(T,x)=u_{T}(x),
	$$
	then the transport equation 
	$$
	m_t - \frac{1}{2}((f_{1}(t)-v_{x})m)_x = 0, \ \ m(0) = m_0
	$$
	possesses at most one weak solution in $L^{1}$.
	\end{lemma}
	The proof of Lemma \ref{uniqueness:continuity:equation} (see e.g. \cite[Section 4.2]{cardaliaguet2010notes}) relies on 
	semi-concavity estimates for the 
	solutions of Hamilton-Jacobi equations \cite{cannarsa2004semiconcave}, and Ambrosio superposition   
	principle \cite{ambrosio2008transport, ambrosio2004transport}. 
	\end{proof}

\bibliographystyle{siam}
\bibliography{C:/mybib/mybib}
\end{document}